\documentclass{amsart}
\usepackage{epsf, graphicx}

\theoremstyle{plain}
\newtheorem{thm}{Theorem}[section]

\newtheorem{lem}[thm]{Lemma}

\theoremstyle{definition}

\newtheorem{rem}[thm]{Remark}

\def\Z{{\mathbb Z}}

\def\R{{\mathbb R}}

\def\1{\hbox{\rm\rlap {1}\hskip.03in{\rom I}}}
\def\Bbbone{{\rm1\mathchoice{\kern-0.25em}{\kern-0.25em}
	{\kern-0.2em}{\kern-0.2em}I}}

\begin{document}
\hyphenation{Ca-m-po}
\title[The number of framings of a knot in a $3$-manifold]
{The number of framings of a knot in a $3$-manifold}
\author[P.~Cahn]{Patricia Cahn}
\address{Department of Mathematics,
David Rittenhouse Lab,
209 South 33rd Street, University of Pennsilvania
Philadelphia, PA 19104-6395, USA }
\email{pcahn@sas.upenn.edu}
\author[V.~Chernov]{Vladimir Chernov}
\address{Department of Matehmatics, 6188 Kemeny Hall, Dartmouth College, Hanover, NH 03755, USA}
\email{Vladimir.Chernov@dartmouth.edu}
\author[R.~Sadykov]{Rustam Sadykov}
\address{Departamento de Matem\'aticas,
Cinvestav-IPN
Av. Instituto Polit\'ecnico Nacional 2508
Col. San Pedro Zacatenco
M\'exico, D.F., C.P.~07360, M\'exico}
\email{rstsdk@gmail.com}

\begin{abstract}
In view of the self-linking invariant, the number $|K|$ of framed knots in $S^3$ with given underlying knot $K$ is infinite. In fact, the second author previously defined affine self-linking invariants and used them to show that $|K|$ is infinite for every knot in an orientable manifold unless the manifold contains a connected sum factor of $S^1\times S^2$; the knot $K$  need not be zero-homologous and the manifold is not required to be compact. 

We show that when $M$ is orientable, the number $|K|$ is infinite unless $K$ intersects a non-separating sphere at exactly one point, in which case $|K|=2$; the existence of a non-separating sphere implies that $M$ contains a connected sum factor of $S^1\times S^2$. 

For knots in nonorientable manifolds we show that if $|K|$ is finite, then $K$ is disorienting, or there is an isotopy from the knot to itself which changes the orientation of its normal bundle, or  it intersects some embedded $S^2$ or $\R P^2$ at exactly one point, or it intersects some embedded $S^2$ at exactly two points in such a way that a closed curve consisting of an arc in $K$ between the intersection points and an arc in $S^2$ is disorienting.  



\end{abstract}
 
\maketitle

\leftline {\em \Small 2000 Mathematics Subject Classification.
Primary: 57M27}

\leftline{\em \Small Keywords: framed knots, self-linking number, Dehn twist}

\section{Introduction} 
We work in the smooth category where the word ``smooth'' means $C^{\infty}.$ Throughout the paper $M$ denotes a manifold of dimension $3$. It is not necessarily compact, orientable or without boundary. 
A \emph{knot} in $M$ is an embedding of a circle $S^1$. A \emph{framing} of a knot is a normal vector field along the knot in $M$.  
An {\em isotopy} of ordinary (respectively, framed) knots is a path in the space of 
ordinary (respectively, framed) embedded connected curves.

For an unframed knot $K$ in $M$ we are interested in the number $|K|$ of isotopy classes of framed knots that correspond to the isotopy class of $K$ when one forgets the framing. Given a framed knot $K_f$ and $i\in \Z$, let $K_f^i$ denote the framed knot obtained by adding $i$ extra positive full-twists to the framing of $K_f$, if $i\geq 0$; and by adding $|i|$ extra negative twists to the framing of $K_f$ if $i<0$. Clearly every isotopy class of framed knots corresponding to $K$ contains $K_f^i$ for some $i\in \Z$. However what could happen, and actually does happen for knots in some manifolds, is that $K_f^i$ and $K_f^j$ are isotopic for $i\neq j$. In this case of course
$|K|$ is finite.

When $K$ is a classical knot in $\R^3$ or, more generally, when $K$ is a zero-homologous knot in an orientable manifold, the number of framings $|K|$ is infinite as  the framed knots $K^i_f$ can be distinguished by the self-linking numbers. On the other hand, self-linking numbers do not exist when the knot is nonzero-homologous. Still intuitively one expects that $|K|=\infty$ for many $K\subset M$. 
The second author~\cite{ChernovFramed1} used the technique of Vassiliev-Goussarov invariants to define the affine self-linking number that generalizes the self-linking number to the case of nonzero-homologous framed knots in an orientable manifold. In particular, the number of framings $|K|$ is infinite for all knots in orientable manifolds that are not realizable as a connected sum $(S^1\times S^2)\#M'$; equivalently, the number of framings of a knot in an orientable manifold $M$ is infinite  unless $M$ contains a nonseparating $2$-sphere. For compact orientable manifolds this result was stated earlier without proof by Hoste and Przytycki in their paper \cite[Remark on page 488]{HostePrzytycki}. They referred to the paper of McCullough~\cite{McCullough1}  on the mapping class group of a $3$-manifold for the idea of the proof.

If $K$ intersects a nonseparating sphere in an orientable manifold at exactly one point then $|K|=2$, see~\cite[page 487]{HostePrzytycki} and later works~\cite[Theorem 2.6]{ChernovFramed1},~\cite[Theorem 4.1.1]{ChernovLegendrian}. Note that for a framed knot $K_f$ in an orientable manifold, the framed knots $K_f$ and $K_f^1$ are not isotopic and $|K|$ is even. Indeed, in order to prove the claim it suffices to observe that the spin structure takes different values on the loops in the principle $\mathop\mathrm{SO}_3$ bundle corresponding to these two framed knots. The orientable $3$-manifold $M$ is parallelizable and the orthonormal 3-framing at a point of a framed knot consists of the framing vector, the normalized
velocity vector of the curve, and the unique unit-length vector orthogonal to them
such that the resulting 3-frame gives the chosen orientation of M. 
This argument admits a generalization for knots in nonorientable manifolds.

\begin{lem}\label{twocomponents}  For a knot $K$ in a (possibly nonorientable) manifold $M$ of dimension $3$, the number $|K|$ is either infinite or even $\geq 2.$ 
\end{lem}
\begin{proof}[Proof of Lemma~\ref{twocomponents}] Let $\xi$ be the fiber bundle over $M$ with fiber over $p\in M$ consisting of
pairs of orthonormal vectors in $T_pM$. It is a principle $\mathop\mathrm{SO}_3$-bundle. Given a framed knot
$K$ in $M$, there is a canonical lift $K'$ of $K$ to the total space $P$ of $\xi$; it is constructed by
associating with $p\in K$ the pair of the velocity vector of $K$ at $p$
and the framing vector of $K$ at $p$. Since the fiber of $\xi$ is
$\mathbb{R}P^3$, there is a canonical one-dimensional vector bundle $E$ over $P$ that
restricts over each fiber to a canonical vector bundle over
$\mathbb{R}P^3$. The vector bundle $E$ can be constructed by means of the Borel construction. Indeed, let $V$ be the total space of the canonical line bundle over $\R P^{3}$. Then $E$ is the total space of the Borel construction
\[
       P\times_{\mathop\mathrm{SO}_3} V\longrightarrow M; 
       \]
in other words $E$ is the quotient of $P\times V$ by the relation generated by equivalences $(xg, y)\sim (x, gy)$ where $g\in \mathop\mathrm{SO}_3$. 
Let $w_1$ denote the first Stiefel-Whitney class of $E$. It is a
non-trivial cohomology class in $P$. The evaluation of $w_1$ on the
lift $K'$ is an invariant of a regular homotopy of $K$. Passing from $K_f$ to $K_f$ with an extra twist of the framing changes the value of $w_1$ by one. 
This implies
the claim.
\end{proof}

We show that the 2006 work of McCullough~\cite{McCullough2} immediately implies the following.

\begin{thm}\label{main}
Let $K$ be a knot in an orientable $3$-manifold $M$.  If $|K|\neq\infty,$ then $K$ intersects some nonseparating $2$-sphere at exactly one point.
\end{thm}

In the case of a non-orientable manifold $M$
we say that a knot $K$ in $M$ is \emph{disorienting} if it is orientation reversing or equivalently if its normal bundle is non-trivial. 

\begin{thm}\label{changeoforientation}
If $K$ is disorienting or if there is an isotopy from $K$ to itself which changes the orientation of its normal bundle, then $|K|=2$. 
\end{thm}
\begin{proof}  If $K$ is disorienting, it suffices to observe that the Dehn twist of a Klein bottle (which is the boundary of the regular neighborhood of $K$) along a meridian in one direction is isotopic to the Dehn twist along the same meridian in the opposite direction. 

If there is an isotopy from $K$ to itself which changes the orientation of the normal bundle, then under this isotopy a knot $K$ with an extra twist of its framing will become a knot with an extra twist of its framing in the opposite direction. 
\end{proof}

{\it In the rest of the section we will consider the case where $K$ is not disorienting and it does not admit an isotopy to itself that reverses the orientation of its normal bundle. \/}

To formulate a counterpart of Theorem 1.1 in this case, we need to introduce an invariant associated with a pair of intersection points between an embedded oriented sphere $S^2\subset M$ and a framed
knot $K_f\subset M$. Without loss of generality, by slightly perturbing
$K_f$ if necessary, we may assume that $K$ intersects $S^2$
transversally. 
Let $p$ and $q$ be two intersection points between $K$ and $S^2$. Choose a closed path $\gamma$ in $M$ that follows from $p$ to $q$ along $K$ and then from $q$ to $p$ along $S^2$. We observe that either for any choice of $\gamma$ it is  disorienting or for any choice of $\gamma$ it is not.


\begin{thm}\label{th:1.4} Let $K$ be a knot in a non-orientable manifold $M$. Suppose that $K$ is not disorienting and there is no isotopy of $K$ to itself which changes the orientation of the normal bundle of $K$ and that $|K|<\infty$. Then either
\begin{description}
\item[a] the knot crosses some embedded two-sided $\R P^2$ at one point; or
\item[b] the knot crosses some embedded $2$-sphere at one point; or
\item[c] the knot $K$ crosses some embedded oriented $2$-sphere $S^2$ at two points and a closed curve consisting of an arc in $K$ between the two intersection points and an arc in $S^2$ is disorienting.  
\end{description}
\end{thm}

Finally, we show that in the cases $[a]$ and $[b]$ the number of framings of $K$ is $2$. 

\begin{thm}\label{l:1.5}
If $K$ intersects transversally at a unique point an embedded sphere or two-sided projective plane, then $|K|=2.$
\end{thm}

\section{Proofs}

\begin{proof}[Proof of Theorem~\ref{main}] To begin with let us assume that $M$ is a compact manifold possibly with boundary. If the number $|K|$ of famings is finite, then there is an isotopy from $K_f$ to $K_f^i$ for some $i.$ Let $U$ denote a thin solid torus neighborhood of $K$ bounded by $T^2=\partial U.$ By the isotopy extension theorem there is a diffeomorphism of $M$ which is identity in a neighborhood of the boundary components of $M$ and which is an $i$-th power of a Dehn twist along the meridian curve $C=C_1$ on $U.$ Put $D_1$ to be a meridian disk in $U$ whose boundary is the meridian curve $C_1.$ Put $M'=M\setminus \mathring{U}$. Then $\partial M'=\partial M\sqcup T^2.$

McCullough~\cite[Theorem 1]{McCullough2} says the following: Let $M'$ be  a compact orientable $3$-manifold which admits a homeomorphism which is Dehn twists on the boundary about a collection $C_1, \dots, C_n$ of simple closed curves in $\partial M'$.  Then for each $i$, either $C_i$ bounds a disk in $M'$, or for some $j\neq i$, $C_j$ and $C_i$ cobound an incompressible annulus in $M'.$

In our case the collection $C_1, \dots, C_n$ of closed simple curves consists of only one curve $C_1$, and therefore $C_1$ bounds a disk $D_2$ in $M'.$ The disks $D_1, D_2$ have a common boundary and together they form a nonseparating $2$-sphere that intersects $K$ transversally at exactly one point. This completes the proof for compact $M.$

When $M$ is noncompact it suffices to notice that the isotopy from $K_f$ to $K_f^i$ is contained within a compact submanifold of $M$, and if this submanifold contains a nonseparating $2$-sphere intersecting $K$ at exactly one point, then so does $M.$
\end{proof}

McCullough states in \cite[ page 1332]{McCullough2} without proof that 
using a more elaborate machinery his~\cite[Theorem 1]{McCullough2} mentioned above can be generalized to nonorientable manifolds adding a possibility of Dehn twists about (two-sided) M\"obius bands. 
Also when $M$ is nonorientable, an annulus could meet the boundary in such a way that a Dehn twist along it is isotopic on the boundary to an even power of a Dehn twist along one of the boundary curves.  (In this case one of the two components of the annulus boundary does not have to be an element of the collection of the curves $C_i$~\cite{McCullough3}.)

\begin{proof}[First proof of Theorem~\ref{th:1.4}] The argument is similar to one in the proof of Theorem~\ref{main} except that now we apply the generalization of the McCullough theorem to the case of nonorientable manifolds.

\begin{figure}[h]
\centering
\includegraphics[height=1.6in]{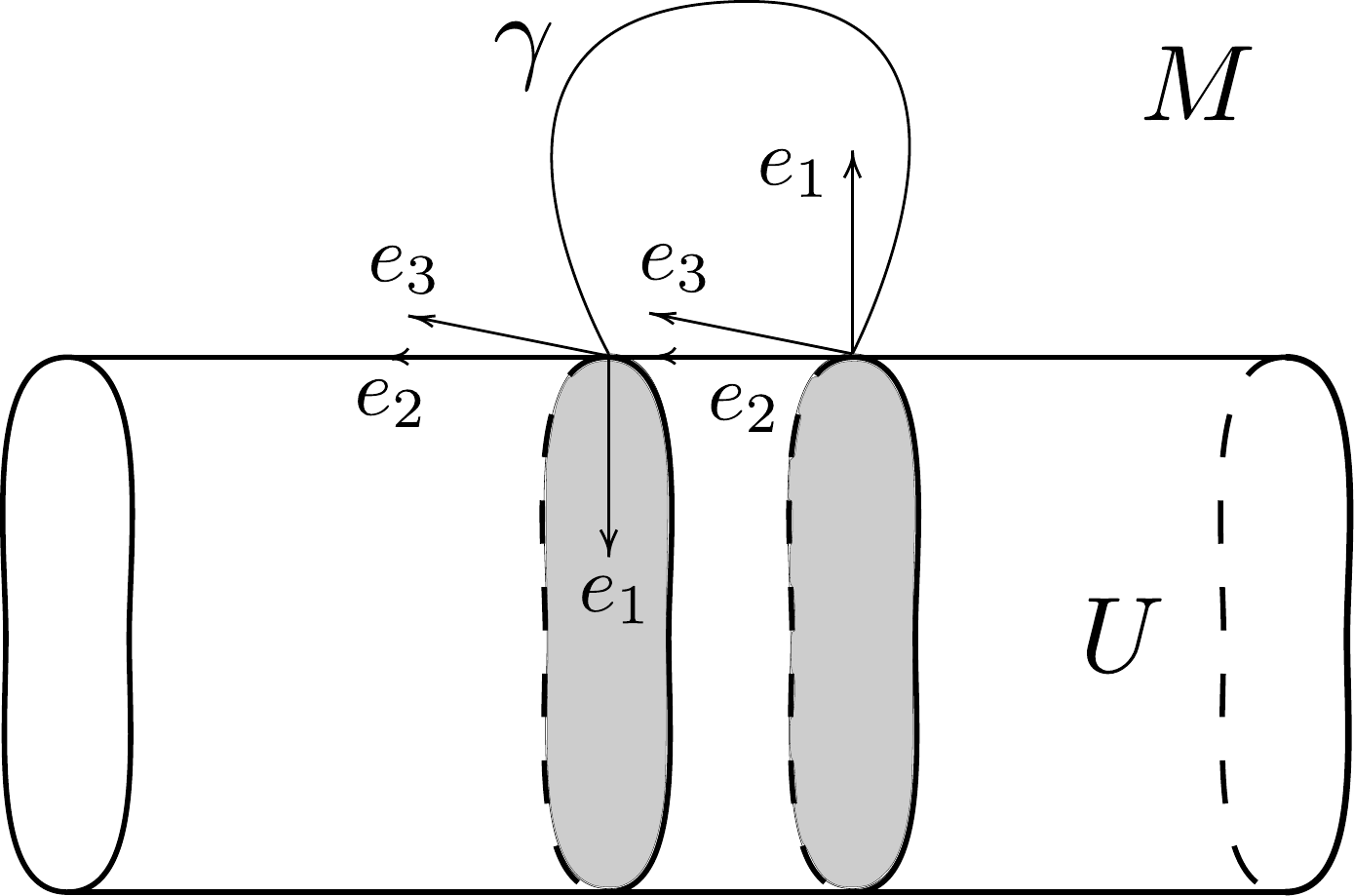}
\caption{Disorienting curve $\gamma$.}
\label{fig:3}
\end{figure}

The case of a M\"obius band bounding $C_i$ in the generalization of the McCullough theorem corresponds to the case $[a]$ in Theorem~\ref{th:1.4}. We may assume that the projective plane $\R P^2$ in $[a]$ is two-sided in $M$ since otherwise it can be replaced with the boundary $S^2\subset M$ of its tubular neighborhood; the bounding sphere $S^2$ satisfies the requirements of $[c]$.  The case of a disc bounding $C_i$ in the generalization of the McCullough theorem corresponds to $[b]$. Finally, let $U$ be a tubular neighborhood of $K$, and suppose that there exists an annulus in $M\setminus \mathring{U}$ that meets the boundary in such a way that a Dehn twist along it is isotopic to an even power of a Dehn twist of $\partial U$. Then, capping off the embedded annulus with two discs in $U$ results in a sphere satisfying the requirements of $[c]$, see Figure~\ref{fig:3}.
\end{proof}

Theorem~\ref{th:1.4} can also be proven using Theorem~\ref{main},
without assuming the generalization of the McCullough theorem to nonorientable manifolds. Since the proof of the generalization of the McCullough Theorem to nonorientable manifolds is not in the literature and appears to be quite involved, we provide an alternative proof below.

\begin{rem}In our proof below we use a lift of $K$ to the orientation double cover of $M$. If  $K$ is disorienting, then the lift disintegrates.

If there is an isotopy of $K$ which changes the orientation of its normal bundle, then the automorphism of $M\setminus \mathring U$ obtained as a result of isotopy extension is not a Dehn twist on the boundary torus of $M\setminus \mathring U$ and we cannot apply Theorem~\ref{main}. 

This is why we considered these two cases separately in Theorem~\ref{changeoforientation}.

\end{rem}
 
\begin{proof}[Second proof of Theorem~\ref{th:1.4}]  
Choose a framed knot $K_f$ corresponding to $K$. Since $|K|<\infty$, there exists an ambient isotopy $h$ of $M$ that takes $K_f$ to itself with a different framing. 

Let $\tilde M$ be the orientation double cover of $M$ and $\pi: \tilde M\to M$  the projection. Since $K$ is not disorienting, $K_f$ has two framed lifts $K_1$ and $K_2$ in $M$. Furthermore, the isotopy $h$ of $M$ taken twice lifts to an isotopy of $\tilde M$ that takes $K_1$ to itself but with a different framing. Since $\tilde M$ is orientable, by the McCullough theorem \cite[Theorem 1]{McCullough2} and the argument in the proof of Theorem~\ref{main}, either  $K_1$ intersects an embedded sphere in $\tilde M\setminus K_2$,  or each of the knots $K_1$ and $K_2$ intersects the same embedded sphere at a unique point. 

Suppose that $K_1$ intersects a sphere $S^2$ in $\tilde M\setminus K_2$. Then $\pi|S^2$ intersects $K$ transversally at a unique point. Therefore, if $U$ is a thin solid torus neighborhood of $K$, then the circle $S^1=\pi S^2\cap \partial U$ is contractible in $M\setminus \mathring{U}$. By the Loop Theorem, there is an embedded disc $D$ in $M\setminus \mathring{U}$ bounding $S^1$. Together with the disc $\pi S^2\cap U$ it forms an embedded sphere in $M$ intersecting $K$ transversally at a unique point. 

Suppose now that the knots $K_1$ and $K_2$ each intersect the same embedded sphere  $S^2\subset \tilde M$ at one point. Then $K$ crosses the projected sphere $\pi S^2$ at two points. By a general position argument, we may assume that the projected sphere is an immersed sphere with at most double points; there are no triple points since $\tilde M\to M$ is a double covering. In what follows we will denote the immersion $\pi|S^2$ by $f$. By the definition of the orientation double cover, a curve consisting of an arc in $K$ and an arc in the immersed sphere $\pi S^2$ is disorienting. 
Our aim is to modify $f$ by surgery to get an embedding of either a sphere or projective plane satisfying one of the three conclusions of Theorem~\ref{th:1.4}. 

Let $C$ be a path component of double points of the immersion $f$ in $M$. 
\begin{figure}[h]
\centering
\includegraphics[height=1.6in]{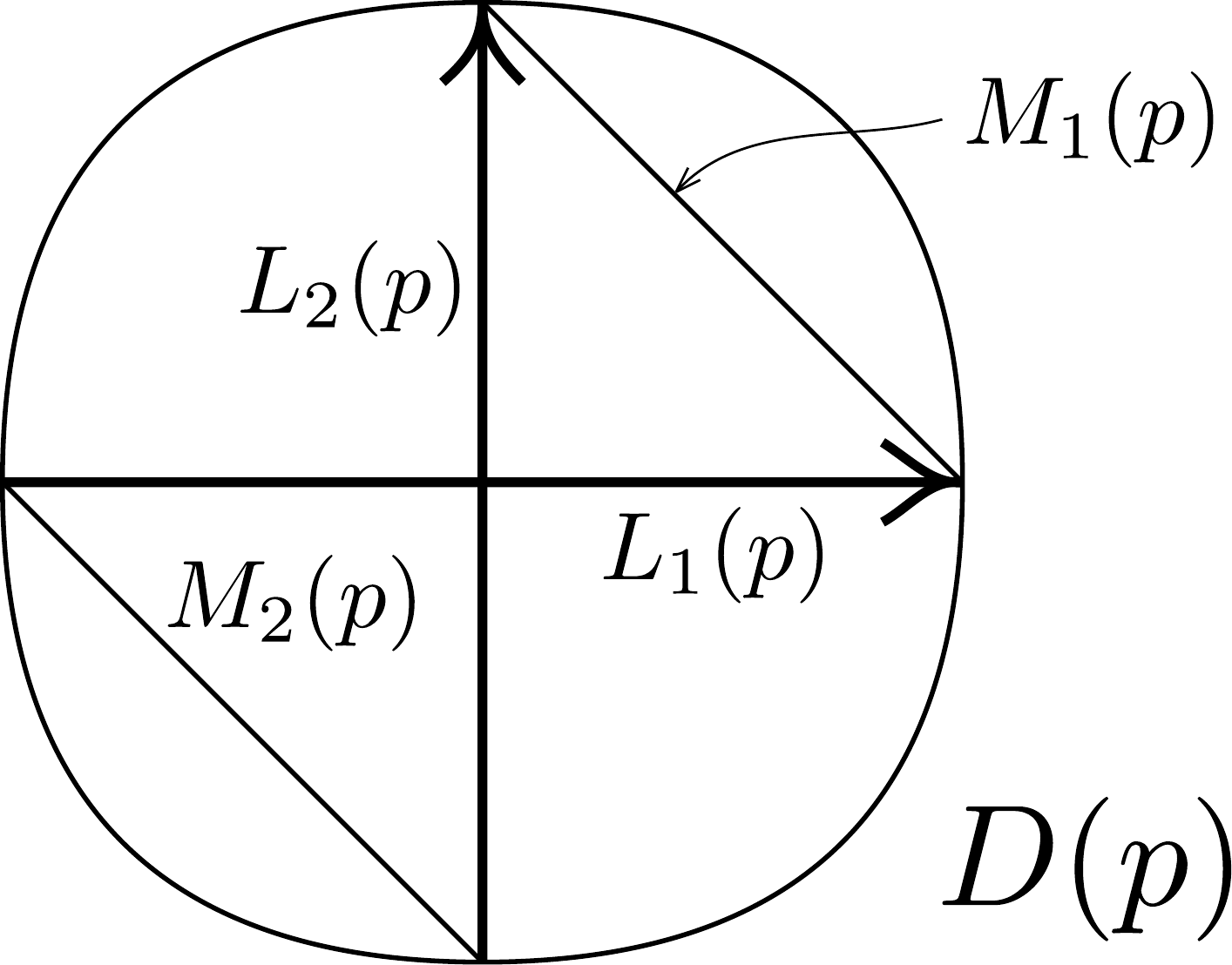}
\caption{Segments $M_1(p)$ and $M_2(p)$.}
\label{fig:1}
\end{figure}

If $f^{-1}C$ consists of one circle, then the tubular neighborhood of $C$ in $M$ is non-orientable. In particular, the component $C$ is non-contractible in $M$, and therefore $f^{-1}C$ is non-contractible in the cylinder $S^2\setminus f^{-1} \mathring U$, where $U$ is a thin solid torus neighborhood of $K$.  
 There is an obvious surgery eliminating the self-intersection circle of $f(S^2)$. Under this surgery we remove the cylindrical neighborhood $V$ of $f^{-1}(C)$ in $S^2$ and attach two M\"obius bands $M_1$ and $M_2$ along the new boundary components of $S^2$. The resulting surface $F$ consists of two copies of $\R P^2$. The new map $f': F\to M$ 
 outside $M_1$ and $M_2$  remains the same as the map $f$ outside $V$. Let us now describe $f|M_1$ and $f|M_2$.  Let $D(p)$ be a normal disc in the tubular neighborhood $U$ of $K$ passing through $p\in K$. The disc $D(p)$ intersects  $f(S^2)$ along two segments $L_1(p)$ and $L_2(p)$ that we can orient so that the monodromy along $K$ exchanges oriented segments $L_1(p)$ and $L_2(p)$. Let $M_1(p)$ and $M_2(p)$ respectively denote the two segments in $D(p)$ joining the terminal and initial points of the oriented segments $L_1(p)$ and $L_2(p)$. We define $f'$ on $M_1$ and $M_2$ so that its image is the union of all segments $M_1(p)$ and $M_2(p)$ where $p$ ranges over all points $p\in K$. 
 
 After the surgery along $C$ we obtain two immersed copies of $\R P^2$. We discard one of them. The remaining immersed  copy of $\R P^2$  
intersects $K$ transversally and the number of components of its double points is less than before the surgery. We observe that the restriction of $f'$ to the M\"obius band $M_i$ in the remaining copy of $\R P^2$ is an embedding. Therefore, if $f'$ has a component $C'$ of double points whose inverse image $f'^{-1}C'$ consists of one circle, then $f'^{-1}C'$ is a non-trivial circle in the cylinder $\R P^2\setminus (M_i\cup f'^{-1}\mathring{U})$ and therefore the argument can be iterated. 

Eventually we obtain an immersion of a surface to $M$---which we continue to denote by $f$---such that for each component $C$ of double points of $f$ the inverse image $f^{-1}(C)$ consists of two circles $C_1$ and $C_2$. The source manifold of $f$ is either $\R P^2$ or $S^2$ depending on wether we performed at least one surgery attaching a M\"obius band or not. 
{\it To simplify notation, we will assume that the source manifold of $f$ is $\R P^2$; the case of $S^2$ is similar.  \/}

As above we deduce that $C_1$ and $C_2$ belong to the cylinder $Q$ obtained from $\R P^2$ by removing $f^{-1}\mathring{U}$ and the attached M\"obius band. There are several possible cases that we need to consider.

\begin{figure}[h]
\centering
\includegraphics[height=1.8in]{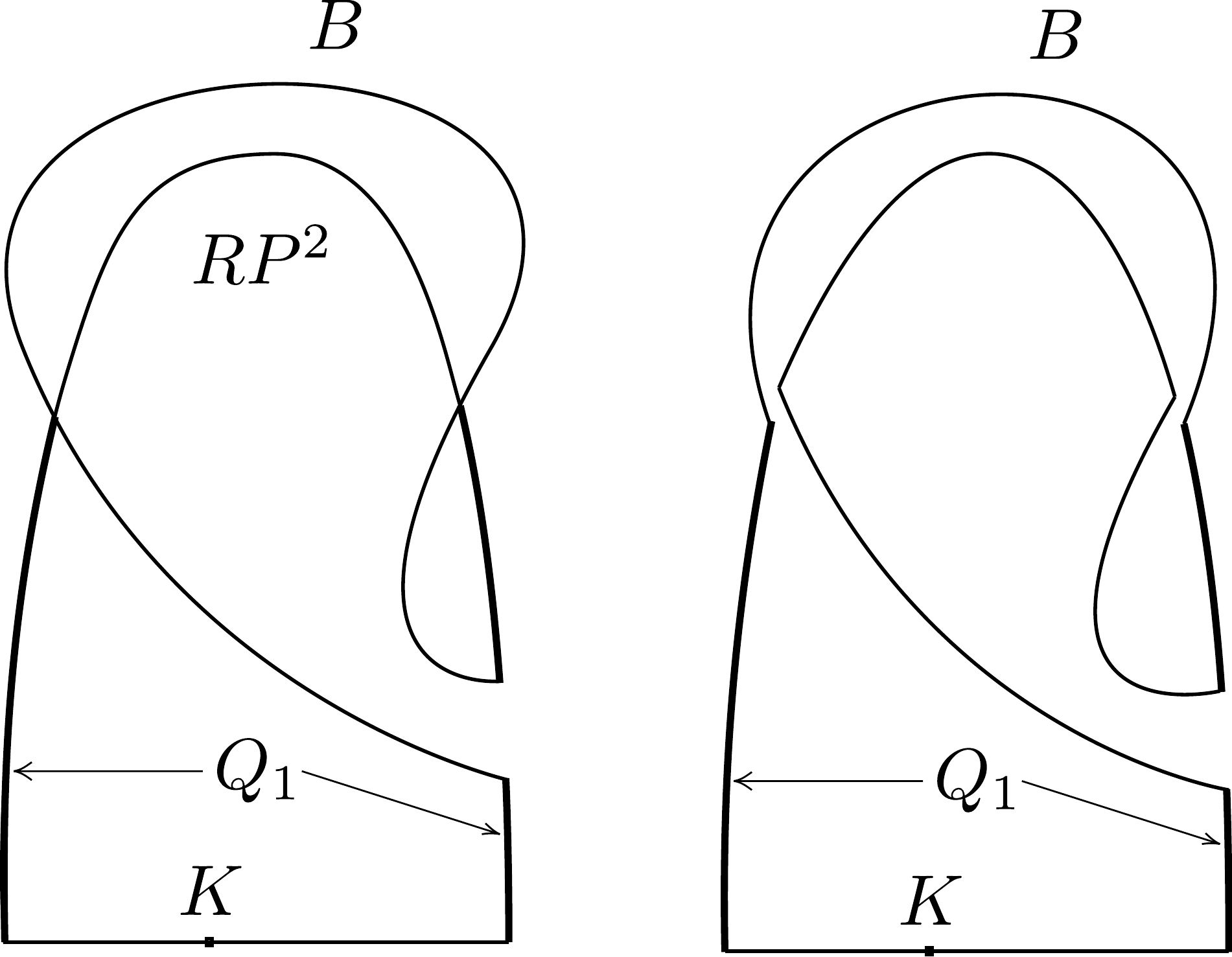}
\caption{The first case.}
\label{fig:2}
\end{figure}

{\it First, if $C_1$ is non-trivial (noncontractible) in $Q$ and $C_2$ is trivial,\/} then $C_1$ breaks the cylinder $Q$ into two closed cylinders $Q_1$ and $Q_2$.
Let $Q_1$ be the cylinder whose boundary intersects the boundary of $f^{-1}(U).$ Now $C_2$ can lie in either $Q_1$ or $Q_2$. 

First suppose $C_2$ is in $Q_1.$ Then we do surgery which exchanges the M\"obius strip bounded by $C_1$ and the disc bounded by $C_2.$ The result is $\R P ^2$ intersected at one point by $K$ but with fewer number of double point components.

Now suppose that $C_2$ is in $Q_2.$ Then we do surgery that glues the disk bounded by $C_2$ to $C_1.$ The result is $S^2$ intersected at one point by $K$ with a fewer number of double point components.


{\it Second, if both $C_1$ and $C_2$ are non-trivial in $Q$,\/} then each of $C_1$ and $C_2$ cut $Q$ into two cylinders.   
There is a surgery resulting in an immersion of a copy of $\R P^2$ intersecting $K$ at a unique point and a closed surface in $M\setminus K$. We discard the closed surface and obtain a new map $f$ with fewer number of components of double points. 

{\it Finally if both $C_1$ and $C_2$ are trivial in $Q$,\/} then $C_1$ cuts $Q$ into a disc $D_1$ and its complement $Q\setminus D_1$, while $C_2$ cuts $Q$ into a disc $D_2$ and its complement $Q\setminus D_2$. There is a surgery resulting in a map of a copy of the cylinder obtained from $Q$ by exchanging the two discs $D_1$ and $D_2$. The result is $\R P ^2$ intersected at one point by $K$ but with fewer number of double point components.

We induct on the number of double point components and after finitely many steps we obtain a desired embedded surface. 
\end{proof}

\begin{rem}  The above proof of Theorem \ref{th:1.4} produces an embedded sphere or projective plane $S$ in $M$ which is the projection of a sphere or projective plane $\tilde S\subset \tilde M$, where $\tilde S$ intersects each lift $K_i$ of $K_f$ at most once.  If $S$ is a sphere then $\tilde S$ is a sphere as well.  If $S$ is a sphere that intersects $K_f$ in two points, it must satisfy the conditions of $[c]$.  Suppose $S$ does not satisfy $[c]$. Then every loop $\gamma$ of the form $\gamma_1\gamma_2$, where $\gamma_1$ is a path from $p$ to $q$ along $K$ and $\gamma_2$ is a path from $q$ to $p$ along $S$, would lift to a loop in $\tilde M$, and hence $S$ would lift to a sphere intersected twice by $K_1$ or twice by $K_2$. Since $\tilde S$ intersects each $K_i$ at most once, this is impossible.

\end{rem}


\begin{proof}[Proof of Theorem~\ref{l:1.5}] 
Let us show that if case $[a]$ occurs then $|K|=2.$ Consider a thin neighborhood of $\R P^2$ that is crossed by $K_f$ at one point. We assume that this neighborhood is $[0,1]\times \R P^2.$ We draw
a curve $\alpha:[0,1]\to \R P ^2$ that is a small figure $8$ equipped with the normal vector field. We assume that $K_f$ intersects $[0,1]\times \R P^2$ as follows: the level $t\times \R P^2$ is intersected at the point $\alpha(t)$ and the direction of the framing vector of $K_f$ at this point projects to $\R P^2$ to the direction of the normal vector to $\alpha$ at $\alpha(t).$ 

Now we construct an isotopy of the knot $K_f$ to the knot with two extra twists of the framing (of the same sign) by describing a homotopy $\alpha_{\tau}, \tau\in [0,1]$ of the curve $\alpha=\alpha_0$. This isotopy will fix the framed knot outside of the set $[\epsilon, 1-\epsilon]\times \R P^2.$ 

Consider the covering $p:S^2\to \R P^2$ and one of the two lifts $\beta$ of $\alpha$ to $S^2$. This $\beta$ is a small figure $8$ having a locally well-defined winding number $0.$ On a sphere the curve with the winding number $0$ is regular homotopic to the curve with the winding number $2.$ Let $\beta_{\tau}, \tau\in[0,1]$ be this homotopy. We assume that it fixes $\beta$ on the neighborhood $\beta([0,\epsilon))\sqcup \beta((1-\epsilon, 1]).$
Put $\alpha_{\tau}=p\circ\beta_{\tau}$. One checks that the framed knot whose intersection with $[0,1]\times \R P^2$ is described by the curve $\alpha_1$ has two more twists of framings of the same sign than the knot $K_f.$
(We can not say whether this knot is $K_f^2$ or $K_f^{-2}$ because $M$ is nonorientable and $K_f^{\pm 2}$ are not well defined.) So $|K|\leq 2$. Since $|K|\geq 2$ by Lemma~\ref{twocomponents}, we have $|K|=2.$

If case $[b]$ occurs the proof is similar and simpler than the one above. Alternatively one can use the generalization of~\cite[page 487]{HostePrzytycki} and later works~\cite[Theorem 2.6]{ChernovFramed1},~\cite[Theorem 4.1.1]{ChernovLegendrian}. 
\end{proof}

{\bf Acknowledgments.}
This paper was written when the first two authors were at Max Planck Institute for Mathematics, Bonn and the authors thank the institute for its hospitality. This work was partially supported by a grant from the Simons Foundation
$\#235674$ to Vladimir Chernov and by a grant CONACYT $\#179823$ to Rustam Sadykov.

The authors would like to thank D. McCullough for valuable discussions about \cite{McCullough2}.


\begin{thebibliography}{99999}


\bibitem{ChernovFramed1}
V.~Chernov (Tchernov): {\em Framed knots in 3-manifolds and affine self-linking numbers,\/} J.~Knot Theory Ramifications {\bf 14} (2005) no.6 pp 791-818 previously available as a preprint math.GT/0105139 at http://www.arxiv.org (2001)

\bibitem{ChernovLegendrian}
V.~Chernov (Tchernov): {\em Vassiliev Invariants of Legendrian, Transverse, and
Framed Knots in Contact $3$-manifolds,\/} Topology 42 (2003), no. 1, 1--33


\bibitem{HostePrzytycki}
J.~Hoste and J.~H.~Przytycki: 
{\em Homotopy skein modules of orientable $3$-manifolds.\/} Math. Proc. Cambridge Philos. Soc. {\bf 108} (1990), no. 3, 475--488.


\bibitem{McCullough1}
D.~ McCullough, {\em  
Mappings of reducible $3$-manifolds,\/} 
Geometric and algebraic topology,  
Banach Center Publ., 18, 
PWN, Warsaw, (1986) pp. 61--76

\bibitem{McCullough2}
D.~McCullough: {\it Homeomorphisms which are Dehn twists on the boundary.\/} Algebr.~Geom.~Topol.~{\bf 6} (2006) 1331--1340

\bibitem{McCullough3}
D.~McCullough: Private communication (2014)

\end{thebibliography}
\end{document}